\documentclass[reqno,12pt]{amsart}
\usepackage{amsmath,epsfig}
\usepackage{amssymb}

\newtheoremstyle{rem}{1.3ex}{1.3ex}{\rmfamily}{}
{\itshape\rmfamily}{}{1.5ex}{}

\newtheorem{theorem}{Theorem}[section]

\theoremstyle{definition}

\newtheorem{remark}[theorem] {Remark}

\renewcommand{\section}{\secdef\sct\sect}
\newcommand{\sct}[2][default]{\refstepcounter{section}
\setcounter{equation}{0}
\vspace{0.5cm}
\centerline{ \large
\scshape \arabic{section}.\ #1}
\vspace{0.3cm}}
\newcommand{\sect}[1]{
\vspace{0.5cm}
\centerline{\large\scshape #1}
\vspace{0.3cm}}

\renewcommand{\subsection}{\secdef \subsct\sbsect}
\newcommand{\subsct}[2][default]{\refstepcounter{subsection}
\nopagebreak
\vspace{0.5\baselineskip}
{\flushleft\bf \arabic{section}.\arabic{subsection}~\bf #1  }
\nopagebreak}
\newcommand{\sbsect}[1]{\vspace{0.1cm}\noindent
{\bf #1}\vspace{0.1cm}}

\def\phi{\varphi }

\newcommand{\C}     {\mathbb{C}}

\newcommand{\R}     {\mathbb{R}}

\newcommand{\N}     {\mathbb{N}}
\renewcommand{\P}   {\mathbb{P}}

\newcommand{\E}     {\mathbb{E}}

\newcommand{\V}     {\mathbb{V}}

\def\1{{\mathchoice {1\mskip-4mu\mathrm l}
                    {1\mskip-4mu\mathrm l}
                    {1\mskip-4.5mu\mathrm l} {1\mskip-5mu\mathrm l}}}

\begin{document}

\title[Edge fluctuations for Wigner matrices]{\large Edge fluctuations of eigenvalues\\\vspace{5mm}of Wigner matrices}

\author[Hanna D\"oring and Peter Eichelsbacher]{} 
\maketitle
\thispagestyle{empty}
\vspace{0.2cm}

\centerline{\sc Hanna D\"oring\footnote{Ruhr-Universit\"at Bochum, Fakult\"at f\"ur Mathematik,
NA 4/31, D-44780 Bochum, Germany, {\tt hanna.doering@ruhr-uni-bochum.de}}, Peter Eichelsbacher\footnote{Ruhr-Universit\"at Bochum, Fakult\"at f\"ur Mathematik,
NA 3/67, D-44780 Bochum, Germany, {\tt peter.eichelsbacher@ruhr-uni-bochum.de}
\\The second author has been supported by Deutsche Forschungsgemeinschaft via SFB/TR 12.}}


\vspace{2 cm}

\begin{quote}
{\small {\bf Abstract:} }
We establish a moderate deviation principle (MDP) for the number of eigenvalues of a Wigner matrix in an interval close 
to the edge of the spectrum. Moreover we prove a MDP for the $i$th largest eigenvalue close to the edge. The proof relies on fine asymptotics of the variance of the eigenvalue counting function of GUE matrices due to Gustavsson.
The extension to large families of Wigner matrices is based on the Tao and Vu Four Moment Theorem. 
Possible extensions to other random matrix ensembles are commented.
\end{quote}




\newpage
\setcounter{section}{0}

\section{Introduction}
Recently, in \cite{Dallaporta/Vu:2011} and \cite{Dalla:2011} the Central Limit Theorem (CLT) for the eigenvalue counting function of {\it Wigner matrices}, 
that is the number of eigenvalues falling in an interval, was established. This {\it universality result} relies on fine asymptotics of the variance of the eigenvalue counting function, on the Fourth Moment Theorem due to Tao and Vu as well as on recent localization results due to Erd\"os, Yau and Yin. There are many
random matrix ensembles of interest, but to focus our discussion and to clear the exposition we shall restrict ourselves to the most famous model
class of ensembles, the Wigner Hermitian matrix ensembles.
For an integer $n \geq 1$ consider an $n \times n$ Wigner Hermitian matrix $M_n = (Z_{ij})_{1 \leq i,j \leq n}$: Consider a family of jointly independent complex-valued random variables $(Z_{ij})_{1 \leq i, j \leq n}$ with $Z_{ji} = \bar{Z}_{ij}$, in particular the $Z_{ii}$ are real valued. For $1 \leq i < j \leq n$ require that the random variables have mean zero and variance one and the $Z_{ij} \equiv Z$ are identically distributed,
and for $1 \leq i=j \leq n$ require that $Z_{ii} \equiv Z'$ are also identically distributed with mean zero and variance one. 
The distributions of $Z$ and $Z'$ are called {\it atom distributions}. An important example of a Wigner Hermitian matrix $M_n$ is the case where the entries are Gaussian, that is $Z_{ij}$ is distributed according
to a complex standard Gaussian $N(0,1)_{\C}$ for $i \not= j$ and $Z_{ii}$ is distributed according to a real standard Gaussian $N(0,1)_{\R}$, 
giving rise to the so-called Gaussian Unitary Ensembles (GUE). GUE matrices will be denoted by $M_n'$. In this case, the joint
law of the eigenvalues is known, allowing a good description of their limiting behavior both in the global and local regimes (see \cite{Zeitounibook}). In the
Gaussian case, the distribution of the matrix is invariant by the action of the group $SU(n)$. The eigenvalues of the matrix $M_n$ are independent
of the eigenvectors which are Haar distributed. If  $(Z_{i,j})_{1 \leq i <j}$ are real-valued the {\it symmetric Wigner matrix} is defined analogously and
the case of Gaussian variables with $\E Z_{ii}^2=2$ is of particular importance, since their law is invariant under the action of the orthogonal group $SO(n)$, known as Gaussian Orthogonal Ensembles (GOE).

The matrix $W_n := \frac{1}{\sqrt{n}} M_n$ is called the coarse-scale normalized Wigner Hermitian matrix, and $A_n := \sqrt{n} M_n$
is called the fine-scale normalized Wigner Hermitian matrix. 
For any $n \times n$ Hermitian matrix $A$ we denote by $\lambda_1(A), \ldots, \lambda_n(A)$ the real eigenvalues of $A$. We introduce the {\it eigenvalue counting function}
$$
N_I(A) := \big| \{ 1 \leq i \leq n : \lambda_i(A) \in I \, \} \big|
$$
for any interval $I \subset \R$. We will consider $N_I(W_n)$ as well as $N_I(A_n)$. Remark that $N_I(W_n) = N_{nI}(A_n)$.
The global {\it Wigner theorem} states that the empirical measure
$\frac 1n \sum_{i=1}^n \delta_{\lambda_i}$
on the eigenvalues of the coarse-scale normalized Wigner Hermitian matrix $W_n$ converges weakly almost surely as $n \to \infty$ to the semicircle law
$$
d \varrho_{sc}(x) = \frac{1}{2 \pi} \sqrt{ 4 - x^2} \, 1_{[-2,2]}(x) \, dx,
$$ 
(see \cite[Theorem 2.1.21, Theorem 2.2.1]{Zeitounibook}). Consequently, for any interval $I \subset \R$,
$$
\lim_{n \to \infty} \frac{1}{n} N_I(W_n) = \varrho_{sc}(I) := \int_I \varrho_{sc}(y) \, dy
$$
almost surely. At the fluctuation level, it is well known that for the GUE, $W_n' := \frac{1}{\sqrt{n}} M_n'$
satisfies a CLT (see \cite{Soshnikov:2000}): Let $I_n$ be an interval in $\R$. If $\V(N_{I_n}(W_n')) \to \infty$
as $n \to \infty$, then
$$
\frac{N_{I_n}(W_n') - \E [ N_{I_n}(W_n')]}{\sqrt{\V (N_{I_n}(W_n'))}} \to N(0,1)_{\R}
$$
as $n \to \infty$ in distribution.
In \cite{Gustavsson:2005} the asymptotic behavior of the expectation and the variance of the counting function $N_I(W_n')$ for intervals $I=[y, \infty)$ with
$y \in (-2,2)$ strictly in the bulk of the semicircle law was established:
\begin{equation} \label{asymp}
\E [ N_{I}(W_n')] = n \varrho_{sc}(I) + O \bigl( \frac{\log n}{n} \bigr) \,\, \text{and} \,\, \V (N_{I}(W_n')) = \bigl( \frac{1}{2 \pi^2} + o(1) \bigr) \, \log n.
\end{equation}
The proof applied strong asymptotics for orthogonal polynomials with respect to exponential weights, see \cite{Deift/Thomas:1999}.
In particular the CLT holds for $N_I(W_n')$ if $I=[y, \infty)$ with $y \in (-2,2)$, and moreover in this case one obtains with \eqref{asymp}
$$
\frac{N_{I}(W_n') - n \varrho_{sc}(I)}{\sqrt{\frac{1}{2 \pi^2} \log n}} \to N(0,1)_{\R}
$$
as $n \to \infty$ (called the CLT with numerics). These conclusions were extended to non-Gaussian Wigner Hermitian matrices in \cite{Dallaporta/Vu:2011}.
\bigskip

\section{Global moderate deviations at the edge of the spectrum}

Certain deviations results and concentration properties for Wigner matrices were considered. Our aim is to establish moderate
deviation principles. Recall that a sequence of laws $(P_n)_{n \geq 0}$ on a Polish space $\Sigma$ satisfies a large deviation principle (LDP)
with good rate function $I : \Sigma \to \R_+$ and speed $s_n$ going to infinity with $n$ if and only if the level sets $\{x: I(x) \leq M\}$, $0 \leq M < \infty$,
of $I$ are compact and for all closed sets $F$
$$
\limsup_{n \to \infty} s_n^{-1} \log P_n(F) \leq - \inf_{x \in F} I(x)
$$
whereas for all open sets $O$
$$
\liminf_{n \to \infty} s_n^{-1} \log P_n(O) \geq - \inf_{x \in O} I(x).
$$
We say that a sequence of random variables satisfies the LDP when the sequence of measures induced by these variables satisfies the LDP. Formally
a moderate deviation principle is nothing else but the LDP. However, we speak about a moderate deviation principle (MDP) for a sequence of random variables,
whenever the scaling of the corresponding random variables is between that of an ordinary Law of Large Numbers (LLN) and that of a CLT.

Large deviation results for the empirical measures of Wigner matrices are still only known for the Gaussian ensembles since their
proof is based on the explicit joint law of the eigenvalues, see \cite{BenArous/Guionnet:1997} and \cite{Zeitounibook}. A moderate deviation
principle for the empirical measure of the GUE or GOE is also known, see  \cite{Dembo/Guionnet/Zeitouni:2003}. This moderate deviations result
does not have yet a fully universal version for Wigner matrices. It has been generalised to Gaussian divisible matrices with a deterministic self-adjoint matrix added with converging empirical measure \cite{Dembo/Guionnet/Zeitouni:2003} and to Bernoulli matrices \cite{DoeringEichelsbacher:2009}.
Recently we proved in \cite{DoeringEichelsbacher:2011} a MDP for the number of eigenvalues of a GUE matrix in an interval. 
If $M_n'$ is a GUE matrix and $W_n' := \frac{1}{\sqrt n} M_n'$ and $I_n$ be an interval in $\R$. If $\V(N_{I_n}(W_n')) \to \infty$
for $n \to \infty$, then, for any sequence $(a_n)_n$ of real numbers such that
$
1 \ll a_n \ll \sqrt{\V (N_{I_n}(W_n'))}
$, 
the sequence $(Z_n)_n$ with
$$
Z_n = \frac{N_{I_n}(W_n') - \E [ N_{I_n}(W_n')]}{a_n \, \sqrt{\V (N_{I_n}(W_n'))}}
$$
satisfies a MDP with speed $a_n^2$ and rate function $I(x)=\frac{x^2}{2}$.
Moreover let $I=[y, \infty)$ with $y \in (-2,2)$ strictly in the bulk, then the sequence $(\hat{Z}_n)_n$ with
$
\hat{Z}_n = \frac{N_{I}(W_n') - n \varrho_{sc}(I)}{a_n \, \sqrt{\frac{1}{2 \pi^2} \log n}}
$
satisfies the MDP with the same speed, the same rate function, and in the regime $1 \ll a_n \ll \sqrt{\log n}$ (called the MDP with numerics; see Theorem 1.1
in \cite{DoeringEichelsbacher:2011}). It follows applying \eqref{asymp}.
In  \cite{DoeringEichelsbacher:2011}, these conclusions were extended to non-Gaussian Wigner Hermitian matrices. 

The first observation in this paper is, that the MDP for $(Z_n)_n$ and $(\hat{Z}_n)_n$, respectively, is not restricted to the bulk of the spectrum.
To state the result, let $\delta >0$ and assume that $y_n \in [-2 + \delta, 2)$ and $n(2 - y_n)^{3/2} \to \infty$ when $n \to \infty$. Then with
\cite[Lemma 2.3]{Gustavsson:2005} the variance of the number of eigenvalues of $W_n'$ in $I_n :=[y_n, \infty)$  satisfies
\begin{equation} \label{edgeasy1}
\V (N_{I_n}(W_n')) = \frac{1}{2 \pi^2} \, \log \bigl(  n(2 - y_n)^{3/2} \bigr) \, (1 + \eta(n)),
\end{equation}
where $\eta(n) \to 0$ as $n \to \infty$. Moreover the expected number of eigenvalues of $W_n'$ in $I_n$, when $y_n \to 2^-$, is given by \cite[Lemma 2.2]{Gustavsson:2005}:
\begin{equation} \label{edgeasy2}
\E (N_{I_n}(W_n')) = \frac{2}{3 \pi} n (2 - y_n)^{3/2} + O(1).
\end{equation}
Hence applying Theorem 1.1 in \cite{DoeringEichelsbacher:2011} we immediately obtain:

\begin{theorem} \label{result1}
Let $M_n'$ be a GUE matrix and $W_n' = \frac{1}{\sqrt{n}} M_n$. Let $I_n = [y_n, \infty)$ where $y_n \to 2^-$ for $n \to \infty$. Assume that $y_n \in[-2 + \delta, 2)$ and $n(2 -y_n)^{3/2} \to \infty$ when $n \to \infty$. Then, for any sequence $(a_n)_n$ of real numbers such that
$1 \ll a_n \ll \sqrt{\V (N_{I_n}(W_n'))}$, the sequence
$\frac{N_{I_n}(W_n') - \E [ N_{I_n}(W_n')]}{a_n \, \sqrt{\V (N_{I_n}(W_n'))}}$
satisfies a MDP with speed $a_n^2$ and rate function $I(x)=\frac{x^2}{2}$.
Moreover the sequence
$$
Z_n := \frac{N_{I_n}(W_n') - \frac{2}{3 \pi} n(2 -y_n)^{3/2}}{a_n \, \sqrt{\frac{1}{2 \pi^2} \log (n(2-y_n)^{3/2})}}
$$
satisfies the MDP with the same speed, the same rate function, and in the regime $1 \ll a_n \ll \sqrt{\log (n(2-y_n)^{3/2})}$ (called the MDP with numerics).
\end{theorem}

For symmetry reasons an analogous result could be formulated for the counting function $N_{I_n}(W_n')$ near the left edge of the spectrum.
\bigskip

\section{Local moderate deviations at the edge of the spectrum}

Under certain conditions on $i$ it was proved in \cite{Gustavsson:2005} that the $i$-th eigenvalue $\lambda_i$ of the GUE $W_n'$ satisfies a CLT.
Consider $t(x) \in [-2,2]$ defined for $x \in [0,1]$ by
$$
x = \int_{-2}^{t(x)} \varrho_{sc}(t) \, dt = \frac{1}{2 \pi} \int_{-2}^{t(x)} \sqrt{4 - x^2} \, dx.
$$
Then for $i=i(n)$ such that $i/n \to a \in (0,2)$ as $n \to \infty$ (i.e. $\lambda_i$ is eigenvalue in the bulk), $\lambda_i(W_n')$ satisfies a CLT:
\begin{equation} \label{CLT-Gu}
X_n := \sqrt{\frac{4 - t(i/n)^2}{2}} \frac{\lambda_i(W_n') - t(i/n)}{\frac{\sqrt{\log n}}{n}} \to N(0,1)
\end{equation}
for $n \to \infty$. Remark that $t(i/n)$ is sometimes called the {\it classical or expected location} of the $i$-th eigenvalue. The standard deviation
is $\frac{\sqrt{\log n}}{\pi \sqrt{2}} \, \frac{1}{n \varrho_{sc}(t(i/n))}$. Note that from the semicircular law, the factor $\frac{1}{n \varrho_{sc}(t(i/n))}$ is the mean
eigenvalue spacing. Informally, \eqref{CLT-Gu} asserts in the GUE case, that each eigenvalue $\lambda_i(W_n')$ typically deviates by $O \bigl( \sqrt{\log n}/ (n \varrho(t(i/n))) \bigr)$ around its classical location. This result can be compared with the so called {\it eigenvalue rigidity property} $\lambda_i(W_n') = t(i/n) +
O(n^{-1 + \varepsilon})$ established in \cite{Erdoes/Yau/Yin:2010}, which has a slightly worse bound on the deviation but which holds with overwhelming probability and for
general Wigner ensembles. See  also discussions in \cite[Section 3]{TaoVu:2012}. We proved in \cite[Theorem 4.1]{DoeringEichelsbacher:2011} 
a MDP for $\bigl(\frac{1}{a_n} X_n\bigr)_n$ with $X_n$ in \eqref{CLT-Gu}, for any $1 \ll a_n \ll \sqrt{\log n}$, with speed $a_n^2$ and rate $x^2/2$. Moreover in
\cite[Theorem 4.2]{DoeringEichelsbacher:2011}, these conclusions were extended to non-Gaussian Wigner Hermitian matrices.
The proofs are achieved by the tight relation between eigenvalues and the counting function
expressed by the elementary equivalence, for $I(y)=[y, \infty)$, $y \in \R$,
\begin{equation} \label{relation}
N_{I(y)}(W_n) \leq n-i \,\, \text{if and only if} \,\, \lambda_i(W_n) \leq y.
\end{equation}
This relation is true for any eigenvalue $\lambda_i(W_n)$, independent of sitting being in the bulk of the spectrum or very close to the right edge
of the spectrum. Hence the next goal is to transport the MDP for the counting function of eigenvalues close to the (right) edge, Theorem \ref{result1}, to a MDP for any singular eigenvalue close to the right edge of the spectrum. Consider
$i=i(n)$ with $i \to \infty$ but $i/n \to 0$ as $n \to \infty$ and define $\lambda_{n-i}(W_n')$ as eigenvalue number $n-i$ in the GUE.  
An example is $i(n) = n - \log n$.
In \cite[Theorem 1.2]{Gustavsson:2005} a CLT was proven, which is
\begin{equation} \label{localrv}
Z_{n,i} := \frac{\lambda_{n-i}(W_n') - \bigl( 2 - \bigl( \frac{3 \pi}{2} \frac in \bigr)^{2/3} \bigr)}{ \operatorname{const} \bigl( \frac{ \log i}{i^{2/3} n^{4/3}} \bigr)^{1/2}} \to N(0,1)_{\R}
\end{equation}
in distribution with $\operatorname{const}= \bigl( (3 \pi)^{2/3} 2^{1/3} \bigr)^{-1/2}$. Remark that the formulation in \cite[Theorem 1.2]{Gustavsson:2005}
is different, since first of all the GUE in \cite{Gustavsson:2005} was defined such that the limiting semicircular law has support $[-1,1]$ and, second
the CLT in \cite{Gustavsson:2005} is formulated for $\lambda_{n-i}(M_n')$ instead of $\lambda_{n-i}(W_n')$. The choice of the asymptotic expectation
and variance in \eqref{localrv} can be explained as follows. Let $g(y_n)$ be the expected number of eigenvalues in $I_n = [y_n, \infty)$. Then with \eqref{relation} 
$$
P \bigl( \lambda_{n-i}(W_n') \leq y_n \bigr) = P \bigl( N_{I_n}(W_n') \leq i \bigr) = P \biggl( \frac{N_{I_n}(W_n') - g(y_n)}{ \V(N_{I_n}(W_n'))^{1/2}} \leq \frac{i -g(y_n)}{\V(N_{I_n}(W_n'))^{1/2}} \biggr).
$$
Trying to apply the CLT for $N_{I_n}$ is choosing $y_n$ such that $\frac{i -g(y_n)}{\V(N_{I_n}(W_n'))^{1/2}} \to x$ for $n \to \infty$, because this
would imply $P \bigl( \lambda_{n-i}(W_n') \leq y_n \bigr) \to \int_{-\infty}^x \varphi_{0,1}(t) \, dt$, where $\varphi_{0,1}(\cdot)$ denotes the density
of the standard normal distribution. The candidate for $y_n$ can be found as in the proof of \cite[Theorem 1.2]{Gustavsson:2005}, with $g(y_n) = \frac{2}{3 \pi}
n(2 -y_n)^{3/2} + O(1)$ and $h(y_n) = \V(N_{I_n}(W_n'))^{1/2} = \frac{1}{\sqrt{2} \pi} \log^{1/2} (n(2-y_n)^{3/2}) + o(\log^{1/2} (n(2-y_n)^{3/2}))$.
Applying the same heuristic as on page 157 in \cite{Gustavsson:2005}, we obtain
$$
y_n \approx  2 - \biggl( \frac{3 \pi}{2} \frac in \biggr)^{2/3} + x \,\,    \biggl( \bigl((3 \pi)^{2/3} 2^{1/3} \bigr)^{-1/2} \, \bigl( \frac{ \log i}{i^{2/3} n^{4/3}} \bigr) \biggr)^{1/2}.
$$
With respect to the statement  in Theorem \ref{result1} one might expect a MDP for $\bigl( \frac{1}{a_n} Z_{n,i} \bigr)_n$ for certain growing
sequences $(a_n)_n$. We have
\begin{eqnarray*}
P \bigl( Z_{n,i}/ a_n \leq x \bigr) & = & P \bigl( \lambda_{n-i}(W_n') \leq y_n(a_n) \bigr) = P \bigl( N_{I_n}(W_n') \leq i \bigr) \\
& = & P \biggl( \frac{N_{I_n}(W_n') - \E(N_{I_n}(W_n'))}{ a_n \V(N_{I_n}(W_n'))^{1/2}} 
\leq \frac{i -\E(N_{I_n}(W_n'))}{a_n \V(N_{I_n}(W_n'))^{1/2}} \biggr)
\end{eqnarray*}
with 
\begin{equation} \label{ynan}
y_n(a_n) := 2 - \biggl( \frac{3 \pi}{2} \frac in \biggr)^{2/3} + a _n \, x \,\,    \biggl( \bigl((3 \pi)^{2/3} 2^{1/3} \bigr)^{-1/2} \, \bigl( \frac{ \log i}{i^{2/3} n^{4/3}} \bigr)\biggr)^{1/2}
\end{equation}
and $I_n =[y_n(a_n), \infty)$. 
Since $i \to \infty$ and $i/n \to 0$ for $n \to \infty$, we have that $y_n(a_n) \to 2^-$ for every $a_n$ such that $a_n \ll \bigl( \log i\bigr)^{1/2}$.
Hence we can apply \eqref{edgeasy2}, that is $\E (N_{I_n}(W_n')) = \frac{2}{3 \pi} n (2 - y_n(a_n))^{3/2} + O(1)$.
With 
$$
2 - y_n(a_n) =  \biggl( \frac{3 \pi}{2} \frac in \biggr)^{2/3} \biggl( 1 - \frac{a_n \, x \, \log^{1/2} i}{(3 \pi / \sqrt{2}) i} \biggr)
$$
by Taylor's expansion we obtain
\begin{equation} \label{tay1}
\frac{2}{3 \pi} n (2 - y_n(a_n))^{3/2} =  i - \frac{1}{\sqrt{2} \pi} a_n \, x \, \log^{1/2} i + o \bigl( a_n \, x \, \log^{1/2} i \bigl),
\end{equation}
and therefore
$i -\E(N_{I_n}(W_n'))= \frac{1}{\sqrt{2} \pi} a_n \, x \, \log^{1/2} i + o \bigl( a_n \, x \, \log^{1/2} i \bigl) + O(1)$.
From \eqref{tay1} we obtain that $n(2 - y_n(a_n))^{3/2} \to \infty$ for $n \to \infty$ for every $a_n \ll \bigl( \log i \bigr)^{1/2}$.
Hence we can apply \eqref{edgeasy1}, that is $\V (N_{I_n}(W_n')) = \frac{1}{2 \pi^2} \, \log \bigl(  n(2 - y_n(a_n))^{3/2} \bigr) \, (1 + o(1))$.
With \eqref{tay1} we get
$$
\V (N_{I_n}(W_n')) =  \biggl( \frac{1}{2 \pi^2} \log \bigl( \frac{3 \pi}{2} i \bigr) + \frac{1}{2 \pi^2} \log \biggl( 1 - \frac{a_n \, x \, (\log i)^{1/2}}{\sqrt{2} \pi i} + 
o \bigl( \frac{a_n \, x \, (\log i)^{1/2}}{i} \bigl) \biggr) \biggr) (1 + o(1)).
$$
Summarizing we have proven that for any growing sequence $(a_n)_n$ of real numbers such that $1 \ll a_n \ll (\log i)^{1/2}$
$$
\frac{i -\E(N_{I_n}(W_n'))}{a_n \V(N_{I_n}(W_n'))^{1/2}} = x + o(1).
$$
By Theorem \ref{result1} we obtain for every $x < 0$ that
$
\lim_{n \to \infty} \frac{1}{a_n^2} \log P \bigl( Z_{n,i}/ a_n \leq x \bigr)  = - \frac{x^2}{2}$.
With $P \bigl(  Z_{n,i}/ a_n \geq x \bigr) = P\bigl( N_{I_n}(W_n') \geq i-1 \bigr)$
the same calculations lead, for every $x >0$, to
\begin{equation} \label{soso}
\lim_{n \to \infty} \frac{1}{a_n^2} \log P \bigl( Z_{n,i}/ a_n \geq x \bigr)  = - \frac{x^2}{2}.
\end{equation}
Next we choose all open intervals $(a,b)$, where at least one of the endpoints is finite and where none of the endpoints is zero. Denote
the family of such intervals by ${\mathcal U}$. Now it follows for each $U=(a,b) \in {\mathcal U}$,
$$
{\mathcal L}_{U}: = \lim_{n \to \infty} \frac{1}{a_n^2} \log P \bigl(  Z_{n,i}/ a_n \in U \bigr) = \left\{ \begin{array}{r@{\quad:\quad}l}
b^2/2 & a < b < 0 \\ 0 & a < 0 < b \\ a^2/2 & 0<a<b \end{array} \right. 
$$
By \cite[Theorem 4.1.11]{Dembo/Zeitouni:LargeDeviations}, $(Z_{n,i}/ a_n)_n$ satisfies a weak MDP (see definition in \cite[Section 1.2]{Dembo/Zeitouni:LargeDeviations}) with speed $a_n^2$ and rate function 
$t \mapsto \sup_{U \in {\mathcal U}; t \in U} {\mathcal L}_U = \frac{t^2}{2}$.
With \eqref{soso}, it follows that $(Z_{n,i}/ a_n)_n$ is exponentially tight (see definition in \cite[Section 1.2]{Dembo/Zeitouni:LargeDeviations}), hence by Lemma 1.2.18 in \cite{Dembo/Zeitouni:LargeDeviations}, 
$(Z_{n,i}/ a_n)_n$ satisfies the MDP with the same speed and the same good rate function. Hence we have proven:

\begin{theorem} \label{result2}
Let $M_n'$ be a GUE matrix and $W_n' = \frac{1}{\sqrt{n}} M_n'$. Consider $i=i(n)$ such that $i \to \infty$ but $i/n \to 0$ as $n \to \infty$.
If $\lambda_{n-i}$ denotes the eigenvalue number $n-i$ in the GUE matrix $W_n'$ it holds that for any sequence $(a_n)_n$ of real numbers
such that  $1 \ll a_n \ll (\log i)^{1/2}$, the sequence $\bigl( \frac{1}{a_n} Z_{n,i} \bigr)_n$ with $Z_{n,i}$ given by \eqref{localrv}
satisfies a MDP with speed $a_n^2$ and rate function $I(x) = \frac{x^2}{2}$.
\end{theorem}
\bigskip

\section{Universal local moderate deviations near the edge}

Our next goal is to check whether the precise distribution of the atom variables $Z_{ij}$ of a Hermitian random matrix $M_n$ are
relevant for the conclusion of the MDP stated in Theorems \ref{result1} and \ref{result2}, so long as they are normalized to have mean zero and variance one,
and are jointly independent on the upper-triangular portion of $M_n$. It is a remarkable feature of our MDP results that they are {\it universal}, hence
the distribution of the atom variables are irrelevant in some sense. The arguments used above relied heavily on the special structure of the GUE ensemble,
in particular on the determinantal structure of the joint probability distribution (see \cite[Theorem 1.1 and 1.3]{DoeringEichelsbacher:2011}) and
on the fine asymptotics of the expectation and the variance of the eigenvalue counting function of GUE presented in \cite{Gustavsson:2005}. We apply
the swapping method due to Tao and Vu, in which one replaces the entries of one Wigner Hermitian matrix $M_n$ with another matrix $M_n'$ which are close
in the sense of matching moments.  This goes back to Lindeberg's exchange strategy for proving the classical CLT, \cite{Lindeberg}, first applied to Wigner matrices in \cite{Chatterjee:2006}. The precise statement of the so called Four Moment Theorem needs some
preparation. We will use the notation as in \cite{TaoVu:2012}.

We say that two complex random variables $\eta_1$ and $\eta_2$ {\it match to order $k$} with $k \in \N$ if
$$
\E \bigl[ \text{Re}(\eta_1)^m \, \text{Im}(\eta_1)^l \bigr] = \E \bigl[ \text{Re}(\eta_2)^m \, \text{Im}(\eta_2)^l \bigr]
$$
for all $m,l \geq 0$ such that $m+l \leq k$. We will consider the case when the real and the imaginary parts of  $\eta_1$ or of $\eta_2$
are independent, then the matching moment condition simplifies to the assertion that $E \bigl[ \text{Re}(\eta_1)^m] = E \bigl[ \text{Re}(\eta_2)^m]$
and $E \bigl[ \text{Im}(\eta_1)^l] = E \bigl[ \text{Im}(\eta_2)^l]$ for all $m,l \geq 0$ such that $m+l \leq k$.

We say that the Wigner Hermitian matrix $M_n$ 
obeys Condition ${\bf(C0)}$ if we have the exponential decay condition
$$
P \bigl( |Z_{ij}| \geq t^C \bigr) \leq e^{-t}
$$
for all $1 \leq i,j \leq n$ and $t \geq C'$, and some constants $C, C'$ independent of $i,j,n$. We say that the Wigner Hermitian matrix $M_n$ 
obeys Condition ${\bf(C1)}$ with constant $C_0$ if one has
$$
\E |Z_{ij}|^{C_0} \leq C
$$
for some constant $C$ independent of $n$. Of course, Condition ${\bf(C0)}$ implies Condition ${\bf(C1)}$ for any $C_0$, but not conversely.
The statement of the Four Moment Theorem for eigenvalues is:

\begin{theorem}[Four Moment Theorem due to Tao and Vu] \label{taovu}
Let $c_0>0$ be a sufficiently small  constant. Let $M_n=(Z_{ij})$ and $M_n'=(Z_{ij}')$ be two $n \times n$ Wigner Hermitian
matrices satisfying Condition ${\bf(C1)}$ for some sufficiently large constant $C_0$. Assume furthermore that for any $1 \leq i <j \leq n$, $Z_{ij}$ and $Z_{ij}'$ match to order 4 and for any $1 \leq i \leq n$, and  $Z_{ii}$
and $Z_{ii}'$ match to order 2. Set $A_n :=\sqrt{n} M_n$ and $A_n' := \sqrt{n} M_n'$, let $1 \leq k \leq n^{c_0}$ be an integer, 
and let $G : {\Bbb R}^k \to {\Bbb R}$ be a smooth function obeying the derivative bounds 
$|\nabla^jG(x)| \leq n^{c_0}$ for all $0 \leq j \leq 5$ and $x \in {\Bbb R}^k$. Then for any $1 \leq i_1 < i_2 \cdots < i_k \leq n$, and
for $n$ sufficiently large we have
\begin{equation} \label{taovu}
|\E \bigl( G(\lambda_{i_1}(A_n), \ldots, \lambda_{i_k}(A_n) ) \bigr) - \E \bigl( G(\lambda_{i_1}(A_n'), \ldots, \lambda_{i_k}(A_n') ) \bigr)| \leq n^{-c_0}.
\end{equation}
\end{theorem}

The preliminary version of this Theorem was first established in the case of bulk eigenvalues and assuming Condition ${\bf(C0)}$, \cite{Tao/Vu:2009}.
Later the restriction to the bulk was removed and the Condition ${\bf(C0)}$ was relaxed to Condition ${\bf(C1)}$
for a sufficiently  large value of $C_0$, \cite{Tao/Vu:2010}. Moreover, a natural question is whether the requirement of four matching moments is necessary.
As far as the distribution of individual eigenvalues $\lambda_i(A_n)$ are concerned, the answer is essentially yes. For this see
the discussions in \cite{TaoVu:2012}.

Applying this Theorem for the special case when $M_n'$ is GUE, we obtain the following MDP:

\begin{theorem} \label{result3}
The MDP for $\bigl( \frac{1}{a_n} Z_{n,i} \bigr)_n$, Theorem \ref{result2}, hold for Wigner Hermitian matrices obeying Condition ${\bf(C1)}$ for a sufficiently large $C_0$, and
whose atom distributions match that of GUE to second order on the diagonal and fourth order off the diagonal. Given $i=i(n)$ such that
$i \to \infty$ and $i/n \to 0$ as $n \to \infty$ we have:  The sequence $\bigl( \frac{1}{a_n} Z_{n,i} \bigr)_n$ with 
\begin{equation} \label{localrv2}
Z_{n,i} := \frac{\lambda_{n-i}(W_n) - \bigl( 2 - \bigl( \frac{3 \pi}{2} \frac in \bigr)^{2/3} \bigr)}{ \operatorname{const} \bigl( \frac{ \log i}{i^{2/3} n^{4/3}} \bigr)^{1/2}} 
\end{equation}
satisfies the MDP for any sequence $(a_n)_n$ of real numbers such that  $1 \ll a_n \ll (\log i)^{1/2}$ with speed $a_n^2$ and rate function $I(x) = \frac{x^2}{2}$.
\end{theorem}

\begin{proof} 
Let $M_n$ be a Wigner Hermitian matrix whose entries satisfy Condition ${\bf(C1)}$ and match the corresponding entries of GUE up to order 4.
Let $i$  be as in the statement of the Theorem, and let $c_0$ be as in Theorem \ref{taovu}. Then \cite[(18)]{Tao/Vu:2009}
says that
\begin{equation} \label{inequ}
P \bigl( \lambda_i(A_n') \in I_{-} \bigr) - n^{-c_0} \leq P \bigl( \lambda_i(A_n) \in I \bigr) \leq P \bigl( \lambda_i(A_n') \in I_{+} \bigr) + n^{-c_0}
\end{equation}
for all intervals $I=[b,c]$, and $n$ sufficiently, where $I_{+} := [b-n^{-c_0/10}, c+n^{-c_0/10}]$
and $I_{-} := [b+n^{-c_0/10}, c-n^{-c_0/10}]$. We present the argument of proof of \eqref{inequ} just to make the presentation more self-contained.
One can find a smooth bump function $G : {\mathbb R} \to {\mathbb R}_+$ which is equal to one on the smaller interval $I$ and vanishes outside
the larger interval $I_+$. It follows that $P \bigl( \lambda_i(A_n) \in I \bigr) \leq \E G(\lambda_i(A_n))$ and 
$\E G(\lambda_i(A_n')) \leq P \bigl( \lambda_i(A_n') \in I_+\bigr)$. One can choose $G$ to obey the condition $|\nabla^j G(x)| \leq n^{c_0}$ for $j=0, \ldots, 5$
and hence by Theorem \ref{taovu} one gets
$$
| \E G(\lambda_i(A_n)) - \E G(\lambda_i(A_n'))| \leq n^{-c_0}.
$$
Therefore the second inequality in \eqref{inequ} follows from the triangle inequality. The first inequality is proven similarly using a bump function which is 1 on $I_-$ and vanishes outside $I$.

Now for $n$ sufficiently large we consider the interval $I_n := [b_n, c_n]$ with
$$
b_n :=  b \, a_n \, n \, \operatorname{const} \bigl( \frac{ \log i}{i^{2/3} n^{4/3}} \bigr)^{1/2} + n \bigl( 2 - \bigl( \frac{3 \pi}{2} \frac in \bigr)^{2/3} \bigr),
$$
$$
c_n :=  c \, a_n \, n \, \operatorname{const} \bigl( \frac{ \log i}{i^{2/3} n^{4/3}} \bigr)^{1/2} + n \bigl( 2 - \bigl( \frac{3 \pi}{2} \frac in \bigr)^{2/3} \bigr)
$$
with $b,c \in {\mathbb R}$,  $b \leq c$ and $\operatorname{const}= \bigl( (3 \pi)^{2/3} 2^{1/3} \bigr)^{-1/2}$.
Then for $\frac{1}{a_n} Z_{n,i}$ defined as in the statement of the Theorem we have
$P \bigl( Z_{n,i}/a_n  \in [b,c] \bigr) = P \bigl( \lambda_{n-i}(A_n) \in I_n \bigr)$.
With \eqref{inequ} and \cite[Lemma 1.2.15]{Dembo/Zeitouni:LargeDeviations} we obtain
$$
\limsup_{n \to \infty} \frac{1}{a_n^2} \log P \bigl( Z_{n,i}/a_n  \in [b,c] \bigr) \leq \max \biggl( \limsup_{n \to \infty} \frac{1}{a_n^2} \log 
P \bigl( \lambda_{n-i}(A_n') \in (I_n)_+ \bigr) ; \limsup_{n \to \infty} \frac{1}{a_n^2} \log n^{-c_0} \biggr).
$$
For the first object we have
$$
\P \bigl( \lambda_{n-i}(A_n') \in (I_n)_+ \bigr)  =  P \biggl( 
\frac{\lambda_{n-i}(A_n') - n  \bigl( 2 - \bigl( \frac{3 \pi}{2} \frac in \bigr)^{2/3} \bigr)}{a_n \, n \, \operatorname{const} \bigl( \frac{ \log i}{i^{2/3} n^{4/3}} \bigr)^{1/2}} \in [b - \eta(n), c + \eta(n)] \biggr)
$$
with $\eta(n) = n^{-c_0/10} \bigl( a_n \, n \, \operatorname{const} \bigl( \frac{ \log i}{i^{2/3} n^{4/3}} \bigr)^{1/2} \bigr)^{-1} \to 0$ as $n \to \infty$.
Since $c_0 >0$ and $ \log n / a_n^2 \to \infty$ for $n \to \infty$ by assumption, applying Theorem \ref{result2} we have
$$
\limsup_{n \to \infty} \frac{1}{a_n^2} \log P \bigl( Z_{n,i}/a_n  \in [b,c] \bigr) \leq - \inf_{x \in [b,c]} \frac{x^2}{2}.
$$
Applying the first inequality in \eqref{inequ} in the same manner we also obtain the upper bound
$$
\limsup_{n \to \infty} \frac{1}{a_n^2} \log P \bigl( Z_{n,i}/a_n  \in [b,c] \bigr) \geq - \inf_{x \in [b,c]} \frac{x^2}{2}.
$$
Finally the argument in the last part of the proof of Theorem \ref{result2} can be repeated to obtain the MDP for $(Z_{n,i}/a_n)_n$.
\end{proof}
\bigskip

\section{Universal global moderate deviations near the edge}

Next we show the MDP for the eigenvalue counting function near the edge of the spectrum
for Wigner Hermitian matrices matching moments with GUE up to order four:

\begin{theorem} \label{result4}
The MDP for $(Z_n)_n$, Theorem \ref{result1}, hold for Wigner Hermitian matrices $M_n$ obeying Condition ${\bf(C1)}$ for a sufficiently large $C_0$, and
whose atom distributions match that of GUE to second order on the diagonal and fourth order off the diagonal. 
Let $W_n = \frac{1}{\sqrt{n}} M_n$,  let $I_n = [y_n, \infty)$ where $y_n \to 2^-$ for $n \to \infty$. Assume that $y_n \in[-2 + \delta, 2)$ and $n(2 -y_n)^{3/2} \to \infty$ when $n \to \infty$. Then the sequence
$$
Z_n = \frac{N_{I_n}(W_n) - \frac{2}{3 \pi} n(2 -y_n)^{3/2}}{a_n \, \sqrt{\frac{1}{2 \pi^2} \log (n(2-y_n)^{3/2})}}
$$
satisfies the MDP with speed $a_n^2$, rate function $x^2/2$ and in the regime $1 \ll a_n \ll \sqrt{\log (n(2-y_n)^{3/2})}$.
\end{theorem}

\begin{proof}
For every $\xi \in {\mathbb R}$ and $k_n$ defined by 
$$
k_n := \xi \, a_n \, \sqrt{\frac{1}{2 \pi^2} \log (n(2-y_n)^{3/2})} + \frac{2}{3 \pi} n(2 -y_n)^{3/2}
$$
we obtain that
$P \bigl( Z_n \leq \xi \bigr) = P \bigl( N_{I_n}(W_n) \leq k_n \bigr)$.
 Hence using \eqref{relation} it follows
$$
P \bigl( Z_n \leq \xi \bigr) = P \bigl( \lambda_{n- k_n}(W_n) \leq y_n \bigr) = P \bigl( \lambda_{n- k_n}(A_n) \leq n \, y_n \bigr).
$$
With  \eqref{inequ} we have
$
P \bigl( \lambda_{n- k_n}(A_n) \leq n \, y_n \bigr) \leq P \bigl( \lambda_{n- k_n}(A_n') \leq n \, y_n + n^{-c_0/10} \bigr) + n^{-c_0}
$
and
$$
P \bigl( \lambda_{n- k_n}(A_n') \leq n \, y_n + n^{-c_0/10} \bigr) = P \bigl( \lambda_{n- k_n}(W_n') \leq y_n + n^{-1-c_0/10} \bigr) = 
P \bigl( N_{J_n}(W_n') \leq k_n \bigr),
$$
where $J_n = [y_n + n^{-1-c_0/10}, \infty)$. With $y_n' := y_n + n^{-1-c_0/10}$ we consider
$$
Z_n' = \frac{N_{J_n}(W_n') - \frac{2}{3 \pi} n(2 -y_n')^{3/2}}{a_n \, \sqrt{\frac{1}{2 \pi^2} \log (n(2-y_n')^{3/2})}}.
$$
In order to apply Theorem \ref{result1} for $(Z_n')_n$, we have to check if $y_n' \to 2^-$ and $n(2 -y_n')^{3/2} \to \infty$ when
$n \to \infty$. For a proof see \cite[Section 2]{Dalla:2011}. We present the arguments just to make the presentation more self-contained. By assumption
we take $y_n \in [-2+ \delta, 2)$ with $y_n \to 2^-$. Suppose that $y_n' > 2$ for some $n$, then $y_n-2 + n^{-1 -c_0/10} > 0$, hence $2-y_n < n^{-1 -c_0/10}$,
which implies $n(2-y_n)^{3/2}  < n \, n^{-3/2 - 3c_0/20}$, but the left hand side is growing by assumption, a contradiction. We have proven $y_n' \to 2^-$.
Moreover we have
$$
(2- y_n')^{3/2} = (2-y_n)^{3/2} \biggl( 1 - \frac{n^{-1 -c_0/10}}{2-y_n} \biggr)^{3/2} = (2-y_n)^{3/2} \biggl( 1 - \frac 32 \frac{n^{-1 -c_0/10}}{2-y_n} +o 
\biggl( \frac{n^{-1 -c_0/10}}{2-y_n}\biggr) \biggr).
$$
Notice that $\frac{n^{-1 -c_0/10}}{2-y_n}= \frac{n^{-1/3 -c_0/10}}{(n(2-y_n)^{3/2})^{2/3}} \to 0$ and $ n(2 -y_n)^{3/2} \to \infty$ when $n \to \infty$ by assumption.
Hence we can apply Theorem \ref{result1}, which is the MDP for $(Z_n')_n$. Summarizing we have
$$
P \bigl( Z_n \leq \xi \bigr) \leq P \bigl( Z_n' \leq \xi_n \bigr) + n^{-c_0}
$$
with 
\begin{eqnarray*}
\xi_n & = & \frac{k_n - \frac{2}{3 \pi} n (2-y_n')^{3/2}}{a_n \sqrt{\frac{1}{2 \pi^2} \log (n(2-y_n')^{3/2})}} \\
& = & \frac{ \frac{2}{3 \pi} n \bigl( (2-y_n)^{3/2} - (2- y_n')^{3/2} \bigr)}{a_n \, \sqrt{\frac{1}{2 \pi^2} \log (n(2-y_n')^{3/2})}} + 
\xi \biggl( \frac{ \log (n(2-y_n)^{3/2})}{\log (n(2-y_n')^{3/2})} \biggr)^{1/2}.
\end{eqnarray*}
We will prove that $\xi_n = \xi + o(1)$. Using the preceding representation we have
$$
 n \bigl( (2-y_n)^{3/2} - (2- y_n')^{3/2} \bigr) = \frac 32 n^{-c_0/10}(2-y_n)^{1/2} + o(n^{-c_0/10}) \to 0
 $$
 and $a_n \, \sqrt{\frac{1}{2 \pi^2} \log (n(2-y_n')^{3/2})} \to \infty$ when $n \to \infty$. Moreover
 $$
\frac{ \log (n(2-y_n)^{3/2})}{\log (n(2-y_n')^{3/2})} = \frac{ \log (n(2-y_n)^{3/2})}{\log (n(2-y_n)^{3/2}) + \frac 32 \log \bigl( 1 - \frac{n^{-1 -c_0/10}}{2-y_n}\bigr)} \to 1.
$$
Applying Theorem \ref{result1}, it follows that
$
\lim_{n \to \infty} \frac{1}{a_n^2} \log P \bigl( Z_n \leq \xi \bigr) = -\frac{\xi^2}{2}
$
for all $\xi < 0$. Similarly we obtain 
for any $\xi >0$ that $\lim_{n \to \infty} \frac{1}{a_n^2} \log P \bigl( Z_n \geq \xi \bigr) = -\frac{\xi^2}{2}$ and the MDP for $(Z_n)_n$ 
follows along the lines of the proof of Theorem \ref{result1}.
\end{proof}

\begin{remark}
In a next step one could ask whether the statement of Theorem \ref{result4} is true also for the sequence
$$
\frac{N_{I_n}(W_n) - \E [ N_{I_n}(W_n)]}{a_n \, \sqrt{\V (N_{I_n}(W_n))}}.
$$
Hence the question is whether the asymptotic behavior of the expectation and the variance of $N_{I_n}(W_n)$ is identical to the one for GUE
matrices, given in \eqref{edgeasy1} and \eqref{edgeasy2}. The answer is yes, but only for Wigner matrices obeying Condition $(\bf{C0})$.
The reason for is that the Four Moment Theorem \ref{taovu} deals with a finite number of eigenvalues, whereas $N_{I_n}(W_n)$ 
involves all the eigenvalues of the Wigner matrix $M_n$. Theorem \ref{taovu} does not give the asymptotics  \eqref{edgeasy1} and \eqref{edgeasy2}
for Wigner matrices. A recent result of Erd\"os, Yau and Yin  \cite{Erdoes/Yau/Yin:2010} describe strong localization of the eigenvalues of Wigner matrices
and this result provides the additional step necessary to obtain \eqref{edgeasy1} and \eqref{edgeasy2} for Wigner matrices $M_n$ obeying Condition $(\bf{C0})$.
The result in \cite{Erdoes/Yau/Yin:2010} is that for $M_n$ being a Wigner Hermitian matrix  obeying Condition $(\bf{C0})$, there is a constant $C>0$ such that
for any $i \{1, \ldots, n\}$
$$
P \bigl( |\lambda_i(W_n)-t(i/n)| \geq (\log n)^{C \log \log n} \min(i,n-i+1)^{-1/3} n^{-2/3} \bigr) \leq n^{-3}.
$$
Along the lines of the proof of \cite[Lemma 5]{Dallaporta/Vu:2011} one obtains \eqref{edgeasy1} and \eqref{edgeasy2}. We will not present the details.
\end{remark}
\bigskip

\section{Further random matrix ensembles}

In this section, we indicate how the preceding results for Wigner Hermitian matrices can be stated and proved for {\it real Wigner symmetric}
matrices. Real Wigner matrices are random symmetric matrices $M_n$ of size $n$ such that, for $i<j$, $(M_n)_{ij}$ are i.i.d. with mean zero and 
variance one, $(M_n)_{ii}$ are i.i.d. with mean zero and variance 2. The case where the entries are Gaussian is the GOE mentioned in the
introduction.  As in the Hermitian case, the main issue is to establish our conclusions for the GOE. 
On the level of CLT, this was developed in \cite{Rourke:2010} by means of the famous {\it interlacing formulas}
due to Forrester and Rains, \cite{Forrester/Rains:2001}, that relates the eigenvalues of different matrix ensembles.
\begin{theorem}[Forrester, Rains, 2001] \label{for}
The following relation holds between GUE and GOE matrix ensembles:
\begin{equation} \label{forrai}
{\rm GUE}_n = {\rm even} \bigl( \rm{GOE}_n \cup {\rm GOE}_{n+1} \bigr).
\end{equation} 
\end{theorem}  
The statement is: Take two independent (!) matrices from the GOE: one of size $n \times n$ and one of size $(n+1) \times (n+1)$. Superimpose
the $2n+1$ eigenvalues on the real line and then take the $n$ even ones. They have the same distribution as the eigenvalues of a $n \times n$
matrix from the GUE. Let  $M_n^{\R}$ denote a GOE matrix and let $W_n^{\R} := \frac{1}{\sqrt{n}} M_n^{\R}$. In \cite[Theorem 4.2]{DoeringEichelsbacher:2011} we have proved a MDP for 
\begin{equation} \label{GOEZn}
Z_n^{\R} := \frac{N_{I_n}(W_n^{\R}) - \E[N_{I_n}(W_n^{\R})]}{a_n \sqrt{\V(N_{I_n}(W_n^{\R}))}}
\end{equation}
for any $1 \ll a_n \ll \sqrt{\V(N_{I_n}(W_n^{\R}))}$, $I_n$ an interval in ${\mathbb R}$, with speed $a_n^2$ and rate $x^2/2$. 
If $M_n^{\C}$ denotes a GUE matrix and $W_n^{\C}$ the corresponding normalized matrix,  the nice consequences of \eqref{forrai} were already suitably developed in \cite{Rourke:2010}: applying Cauchy's interlacing theorem one can write
\begin{equation} \label{interl}
N_{I_n}(W_n^{\C}) = \frac 12 \bigl[  N_{I_n}(W_n^{\R}) +  N_{I_n}(\widehat{W}_n^{\R}) + \eta_n'(I_n) \bigr],
\end{equation}
where one obtains ${\rm GOE}_n'$ in $N_{I_n}(\widehat{W}_n^{\R})$ from ${\rm GOE}_{n+1}$ by considering the principle sub-matrix of ${\rm GOE}_{n+1}$
and $\eta_n'(I_n)$ takes values in $\{-2,-1,0,1,2\}$. Note that  $N_{I_n}(W_n^{\R})$ and $N_{I_n}(\widehat{W}_n^{\R})$ are independent because ${\rm GOE}_{n+1}$\
and ${\rm GOE}_{n}$ denote independent matrices from the GOE. Now the same arguments as in \cite[Section 4]{DoeringEichelsbacher:2011} 
and Theorem \ref{result1} lead
to the MDP for $(Z_n^{\R})_n$, if we consider intervals $I_n=[y_n, \infty)$ where $y_n \to 2^-$ for $n \to \infty$. Remark that the interlacing formula \eqref{interl} leads to $2 \V(N_{I_n}(W_n^{\C})) +O(1) = \V(N_{I_n}(W_n^{\R}))$ if $\V(N_{I_n}(W_n^{\C})) \to \infty$. Next the proof of Theorem \ref{result2} can be adapted to obtain an MDP for $\lambda_{n-i}(W_n^{\R})$: 
Consider
$$
Z_{n,i}^{\R} := \frac{\lambda_{n-i}(W_n^{\R}) - \bigl( 2 - \bigl( \frac{3 \pi}{2} \frac in \bigr)^{2/3} \bigr)}{ \operatorname{const} \bigl( \frac{ 2 \log i}{i^{2/3} n^{4/3}} \bigr)^{1/2}}.
$$
With $\E[N_{I_n}(W_n^{\R})] = \E[N_{I_n}(W_n^{\C})] + O(1)$ and $2 \V(N_{I_n}(W_n^{\C})) +O(1) = 
\V(N_{I_n}(W_n^{\R}))$ if $\V(N_{I_n}(W_n^{\C})) \to \infty$ we get a MDP along the lines of the proof of Theorem \ref{result2}. We omit the details. 
The Four Moment Theorem also applies for real symmetric matrices. The proof of the next Theorem is nearly identical to the proofs of Theorem
\ref{result3} and Theorem \ref{result4}.  

\begin{theorem} \label{result5}
Consider a real symmetric Wigner matrix $W_n = \frac{1}{\sqrt{n}} M_n$ whose entries satisfy Condition ${\bf (C1)}$ and match the corresponding entries
of GOE up to order 4. Consider $i=i(n)$ such that $i \to \infty$ and $i/n \to 0$ as $n \to \infty$. Denote the $i$th eigenvalue of $W_n$ by 
$\lambda_i(W_n)$. Let $(a_n)_n$ be a sequence of real numbers such that $1 \ll a_n \ll \sqrt{\log i}$.
Then the sequence $(Z_{n,i})_n$ with
$$
Z_{n,i} = \frac{\lambda_{n-i}(W_n) - \bigl( 2 - \bigl( \frac{3 \pi}{2} \frac in \bigr)^{2/3} \bigr)}{ \operatorname{const} \bigl( \frac{ 2 \log i}{i^{2/3} n^{4/3}} \bigr)^{1/2}}
$$
universally satisfies a MDP with speed $a_n^2$ and rate function $I(x)=\frac{x^2}{2}$. Moreover the statement of Theorem \ref{result4} can be adapted and proved analogously.
\end{theorem}

Remark that one could consider the Gaussian Symplectic Ensemble (GSE). Quaternion self-dual Wigner Hermitian matrices have not been studied. 
Due to Forrester and Rains, the following relation holds between matrix ensembles:
${\rm GSE}_n = {\rm even} \bigl({\rm GOE}_{2n+1} \bigr) \frac{1}{\sqrt{2}}$. The multiplication by $\frac{1}{\sqrt{2}}$ denotes scaling the $(2n+1) \times (2n+1)$ GOE matrix by the factor $\frac{1}{\sqrt{2}}$.
Let $x_1 < x_2 < \cdots < x_n$ denote the ordered eigenvalues of an $n \times n$ matrix from the GSE and let $y_1 <y_2 < \cdots < y_{2n+1}$ denote
the ordered eigenvalues of an $(2n+1) \times (2n+1)$ matrix from the GOE. Then it follows that $x_i = y_{2i}/\sqrt{2}$ in distribution. Hence
the MDP for the $i$-th eigenvalue of the GSE follows from the MDP in the GOE case. We omit formulating the result.




\newcommand{\SortNoop}[1]{}\def\cprime{$'$} \def\cprime{$'$}
  \def\polhk#1{\setbox0=\hbox{#1}{\ooalign{\hidewidth
  \lower1.5ex\hbox{`}\hidewidth\crcr\unhbox0}}}
\providecommand{\bysame}{\leavevmode\hbox to3em{\hrulefill}\thinspace}
\providecommand{\MR}{\relax\ifhmode\unskip\space\fi MR }
\providecommand{\MRhref}[2]{%
  \href{http://www.ams.org/mathscinet-getitem?mr=#1}{#2}
}
\providecommand{\href}[2]{#2}

\end{document}